\documentclass[a4paper,11pt]{amsart}
\usepackage{latexsym}
\usepackage{amsmath,amssymb}
\usepackage{amsmath}

\usepackage[all]{xy}
\usepackage{enumerate}
\usepackage[usenames]{color}
\usepackage{xcolor}

\newtheorem{theo}{Theorem}[section]
\newtheorem{coll}[theo]{Corollary}

\newtheorem{defn}[theo]{Definition}
\newtheorem{ex}[theo]{Example}
\newtheorem{rem}[theo]{Remark}

\newcommand{\Hom}{{\rm Hom}}

\newcommand{\HOM}[3]{\mathrm{HOM}_{#1}(#2,#3)}

\newcommand{\A}{\mathcal A}
\newcommand{\B}{\mathcal B}
\newcommand{\C}{\mathcal C}

\begin{document}
\sloppy

\title[Transfer of CS-Rickart and dual CS-Rickart properties]{Transfer of CS-Rickart and dual CS-Rickart properties 
via functors between abelian categories}

\author[S. Crivei]{Septimiu Crivei}

\address{Faculty of Mathematics and Computer Science, Babe\c s-Bolyai University, Str. M. Kog\u alniceanu 1,
400084 Cluj-Napoca, Romania} \email{crivei@math.ubbcluj.ro}

\author[S.M. Radu]{Simona Maria Radu}

\address{Faculty of Mathematics and Computer Science, Babe\c s-Bolyai University, Str. M. Kog\u alniceanu 1,
400084 Cluj-Napoca, Romania} \email{simonamariar@math.ubbcluj.ro}



\begin{abstract} We study the transfer of (dual) relative CS-Rickart properties via functors between abelian categories. We consider fully faithful functors as well as adjoint pairs of functors. We give several applications to Grothendieck categories and, in particular, to (graded) module and comodule categories.
\end{abstract}

\date{April 2, 2020}

\maketitle

\section{Introduction}

CS-Rickart objects in abelian categories have been introduced and studied by us in \cite{CR1} 
as a common generalization of Rickart objects and extending objects. We have also considered dual CS-Rickart objects,
which generalize dual Rickart objects and lifting objects. But these dual results follow automatically
by the duality principle in abelian categories, showing the advantage of working at this level of generality. 
In order to point out some previous interest in such topics, we mention 
Rickart and dual Rickart objects studied by Crivei, K\"or and Olteanu \cite{CK,CO}, which have been proved to be useful 
in the study of regular objects in abelian categories in the sense of 
D\u asc\u alescu, N\u ast\u asescu, Tudorache and D\u au\c s \cite{DNTD}.
They generalize the module-theoretic notions of Rickart and dual Rickart modules in the sense of 
Lee, Rizvi and Roman \cite{LRR10,LRR11}, and in particular, Baer and dual Baer modules studied by Rizvi and Roman 
\cite{RR04,RR09} 
and Keskin T\"ut\"unc\"u and Tribak \cite{KT} respectively. Extending modules (also called CS-modules) 
and lifting modules have been intensively investigated for the last decades, 
due to their important applications to ring and module theory (e.g., see \cite{CLVW,DHSW}).
Let us also note that CS-Rickart objects and dual CS-Rickart objects
extend the module-theoretic concepts of CS-Rickart and dual CS-Rickart modules studied by Abyzov, Nhan and Quynh 
\cite{AN1,ANQ}. In our approach we needed to develop specific categorical techniques in order to generalize results on 
(dual) CS-Rickart modules and extending (lifting) modules to abelian categories.
The reader is referred to \cite{CR1} for further information and motivation on the topic. 

The paper \cite{CR1} contains illustrating examples, and studies direct summands of (dual) CS-Rickart objects, 
(co)products of (dual) CS-Rickart objects as well as classes all of whose objects are (dual) self-CS-Rickart. The 
present paper is a companion of \cite{CR1}, and investigates the transfer of (dual) relative CS-Rickart properties via 
functors between abelian categories. We consider fully faithful functors and adjoint pairs of functors between abelian 
categories, and we discuss the cases of coreflective and reflective abelian full subcategories of abelian categories
as well as adjoint triples of functors and recollements between abelian categories. We present several applications to 
Grothendieck categories and, in particular, to (graded) module and comodule categories. Also, we derive consequences 
for endomorphism rings of (graded) modules and comodules.

\section{Transfer of (dual) CS-Rickart properties via fully faithful functors}

Let $\mathcal{A}$ be an abelian category. For every morphism $f:M\to N$ in $\mathcal{A}$ we denote by
${\rm ker}(f):{\rm Ker}(f)\to M$, ${\rm coker}(f):N\to {\rm Coker}(f)$, 
${\rm coim}(f):M\to {\rm Coim}(f)$ and ${\rm im}(f):{\rm Im}(f)\to N$ the kernel, the cokernel, 
the coimage and the image of $f$ respectively. Since $\A$ is abelian, we have ${\rm Coim}(f)\cong {\rm Im}(f)$. 
Recall that a morphism $f:A\to B$ is called a \emph{section} (\emph{retraction}) if there is a morphism
$f':B\to A$ such that $f'f=1_A$ ($ff'=1_B$).

We also need the following notions. 

\begin{defn}[{\cite[Definition~2.2]{CR1}}] \rm Let $\mathcal{A}$ be an abelian category. 
\begin{enumerate}
\item \rm A monomorphism $f:M \to N$ in $\mathcal{A}$ is called \textit{essential} 
if for every morphism $h:N \to P$ in $\mathcal{A}$ such that $hf$ is a monomorphism, $h$ is a monomorphism.
\item \rm An epimorphism $f:M \to N$ in $\mathcal{A}$ is called \textit{superfluous} 
if for every morphism $h:P \to M$ in $\mathcal{A}$ such that $fh$ is an epimorphism, $h$ is an epimorphism. 
\end{enumerate}             
\end{defn}

\begin{rem} \label{r:es} \rm As in module categories \cite[17.2, 19.2]{Wis}, 
in the definition of essential monomorphisms one may restrict to epimorphisms $h:N\to P$,
while in the definition of superfluous epimorphisms one may restrict to monomorphisms $h:P\to M$. 
\end{rem}

\begin{defn}[{\cite[Definition~2.4]{CR1}}] \rm Let $M$ and $N$ be objects of an abelian category $\mathcal{A}$. 
Then $N$ is called:
\begin{enumerate}
\item \rm {\textit{$M$-CS-Rickart}} if for every morphism $f:M\to N$ there are an 
essential monomorphism $e:{\rm Ker}(f) \to L$ and a section $s:L \to M$ in $\mathcal{A}$ such that ${\rm ker}(f)=se$.
Equivalently, $N$ is $M$-CS-Rickart if and only if for every morphism $f:M\to N$, 
${\rm Ker}(f)$ is essential in a direct summand of $M$.
\item \rm{\textit{dual $M$-CS-Rickart}} if for every morphism $f:M\to N$ there are a    
retraction $r:N \to P$ and a superfluous epimorphism $t:P\to{\rm Coker}(f)$ in $\mathcal{A}$ such that ${\rm 
coker}(f)=tr$.
Equivalently, $N$ is dual $M$-CS-Rickart if and only if for every morphism $f:M\to N$, 
${\rm Im}(f)$ lies above a direct summand of $N$,
in the sense that ${\rm Im}(f)$ contains a direct summand $K$ of $M$ such that ${\rm Im}(f)/K$ is superfluous in $M/K$.
\item \rm {\textit{self-CS-Rickart}} if $N$ is $N$-CS-Rickart.
\item \rm  {\textit{dual self-CS-Rickart}} if $N$ is dual $N$-CS-Rickart.
\end{enumerate}
\end{defn}

In general, a functor between abelian categories need not preserve or reflect CS-Rickart properties, 
as we may see in the following examples.  

\begin{ex} \rm Consider the ring $R=\mathbb{Z}_2\oplus \mathbb{Z}_{16}$
and the forgetful covariant functor $F:{\rm Mod}(R)\to {\rm Mod}(\mathbb{Z})$ between module categories.
Then $F$ is an exact faithful functor. The endomorphism ring of the right $R$-module $R$ is ${\rm End}_R(R_R)\cong R$,
while the endomorphism ring of the $\mathbb{Z}$-module $(R,+)$ is 
${\rm End}_{\mathbb{Z}}(R,+)\cong \left(\begin{smallmatrix} \mathbb{Z}_2 & \mathbb{Z}_2 \\ 
\mathbb{Z}_2 & \mathbb{Z}_{16}\end{smallmatrix}\right)$. 
Since the rings ${\rm End}_R(R_R)$ and ${\rm End}_{\mathbb{Z}}(R,+)$ are not isomorphic,
it follows that the functor $F$ is not full. Following \cite[p.~64]{Lam2}, 
a ring $R$ is called right self-injective if the right $R$-module $R$ is self-injective. 
Note that $\mathbb{Z}_2$ and $\mathbb{Z}_{16}$ are right self-injective rings \cite[Corollary~3.13]{Lam2}.
Then $R=\mathbb{Z}_2\oplus \mathbb{Z}_{16}$ is also a right self-injective ring \cite[Corollary~3.11B]{Lam2}.
Hence the right $R$-module $R$ is extending \cite[Corollary~6.80]{Lam2}, and so it is self-CS-Rickart.  
But the $\mathbb{Z}$-module $F(R)=\mathbb{Z}_2\oplus \mathbb{Z}_{16}$ is not 
self-CS-Rickart \cite[Example~3.9]{CR1}. 
\end{ex}

\begin{ex} \rm For every $\mathbb{Z}$-module $M$, denote by $t(M)$ the largest torsion submodule of $M$ and 
by $d(M)$ the largest divisible (i.e., injective) submodule of $M$. Let $\A$ be the category of $\mathbb{Z}$-modules, 
and let $\B$ be the category of torsion $\mathbb{Z}$-modules. Note that both $\A$ and $\B$ are abelian categories. 
Consider the covariant functor $F:\A\to \B$ defined by $F(M)=t(d(M))$ on objects $M$ of $\A$ and accordingly on 
homomorphisms. 
Then $F$ is a left exact functor. Also, $F$ is a full, but not faithful functor \cite[Example~4.1]{CKT20}. 
Consider the $\mathbb{Z}$-module $M=\mathbb{Z}_2\oplus \mathbb{Z}_{16}\oplus \mathbb{Z}(p^{\infty})$ for some prime 
$p$. 
Then $F(M)=\mathbb{Z}(p^{\infty})$ is a torsion injective $\mathbb{Z}$-module, 
hence it is an injective object in the category $\B$. Moreover, it is well known that 
every object of $\B$ has an injective envelope, namely the injective envelope of a torsion $\mathbb{Z}$-module $A$ 
is $t(E(A))$, where $E(A)$ is the injective envelope of $A$ in $\A$. Then the injective object $F(M)$ of $\B$ 
is a self-CS-Rickart object in $\B$ \cite[Corollary~5.2]{CR1}.  
By \cite[Corollary~3.2]{CR1} and \cite[Example~3.9]{CR1}, the $\mathbb{Z}$-module $M$ is not self-CS-Rickart, 
because its direct summand $\mathbb{Z}_2\oplus \mathbb{Z}_{16}$ is not self-CS-Rickart.
\end{ex}

We continue with a result on preservation and reflection of (dual) relative CS-Rickart properties via functors.
For a covariant functor $F:\mathcal{A}\to \mathcal{B}$, we denote by ${\rm Im}(F)$ the essential image of $F$,
which consists of all objects $B$ of $\B$ such that $B\cong F(A)$ for some object $A$ of $\A$.

\begin{theo} \label{t:ffess} Let $F:\mathcal{A}\to \mathcal{B}$ be a fully faithful covariant functor between abelian 
categories.
Let $M$ and $N$ be objects of $\A$.
\begin{enumerate}[(i)]
\item Assume that ${\rm Im}(F)$ is closed under kernels or cokernels. 
\begin{enumerate}[(1)]
\item If $F$ is left exact and $N$ is $M$-CS-Rickart, then $F(N)$ is $F(M)$-CS-Rickart.
\item If $F$ is right exact and $N$ is dual $M$-CS-Rickart, then $F(N)$ is dual $F(M)$-CS-Rickart.
\end{enumerate}
\item Assume that ${\rm Im}(F)$ is closed under direct summands. 
\begin{enumerate}[(1)]
\item If $F$ is left exact and $F(N)$ is $F(M)$-CS-Rickart, then $N$ is $M$-CS-Rickart.
\item If $F$ is right exact and $F(N)$ is dual $F(M)$-CS-Rickart, then $N$ is dual $M$-CS-Rickart.
\end{enumerate}
\end{enumerate}
\end{theo}

\begin{proof} (i) (1) Let $g:F(M)\to F(N)$ be a morphism in $\B$. We have 
$g=F(f)$ for some morphism $f:M\to N$ in $\A$, because $F$ is full. Since $N$ is $M$-CS-Rickart, 
there are an essential monomorphism $e:{\rm Ker}(f)\to Q$ 
and a section $s:Q\to M$ in $\mathcal{A}$ such that ${\rm ker}(f)=se$. 
Since $F$ is left exact we have ${\rm ker}(g)=F({\rm ker}(f))=F(s)F(e)$, and then $F(s)$ is a section, and $F(e)$ is a monomorphism.  
We claim that $F(e)$ is an essential monomorphism. To this end, denote $K={\rm Ker}(f)$,
and let $h':F(Q)\to Z$ be a morphism in $\B$ such that $h'F(e)$ is a monomorphism.

Assume first that ${\rm Im}(F)$ is closed under kernels. Let $u':{\rm Ker}(h')\to F(Q)$ be the inclusion morphism.
Now we have the following commutative diagram with exact rows:
$$\SelectTips{cm}{}
\xymatrix{
0 \ar[r] \ar[d] & F(K) \ar[rr]^-{h'F(e)} \ar[d]^{F(e)} & & Z \ar@{=}[d] \\
{\rm Ker}(h') \ar[r]_-{u'} & F(Q) \ar[rr]_-{h'} & & Z
}
$$
in which the left square is a pullback. We have ${\rm Ker}(h')\cong F(A)$ for some object $A$ of $\A$, 
and so $u'=F(u)$ for some monomorphism $u:A\to Q$ in $\A$, because $F$ is fully faithful. 
Let $v={\rm coker}(u):Q\to X$. The faithful functor $F$ reflects pullbacks (e.g., see \cite[Chapter~2, Theorem~7.1]{M}),
hence we obtain the following commutative diagram with exact rows:
$$\SelectTips{cm}{}
\xymatrix{
0 \ar[r] \ar[d] & K \ar[r]^-{ve} \ar[d]^{e} & X \ar@{=}[d] \\
A \ar[r]_-{u} & Q \ar[r]_-{v} & X
}
$$
in which the left square is a pullback. Since ${\rm Ker}(ve)=0$, $ve$ is a monomorphism. 
Then $v$ is a monomorphism, because $e$ is an essential monomorphism. Hence $A=0$, 
and so ${\rm Ker}(h')=0$. Then $h'$ is a monomorphism, which shows that $F(e)$ is an essential monomorphism.

Assume now that ${\rm Im}(F)$ is closed under cokernels (epimorphisms). By Remark \ref{r:es}  
we may take $h':F(Q)\to L'$ to be an epimorphism. Now $L'\cong F(L)$ for some object $L$ of $\A$. 
Since $F$ is full, we have $h'=F(h)$ for some morphism $h:Q\to L$ in $\A$. 
Then $F(he)$ is a monomorphism in $\B$, and so $he$ is a monomorphism in $\A$, because $F$ is faithful. Since $e$ is an essential monomorphism, $h$ is a monomorphism. 
It follows that $h'=F(h)$ is a monomorphism since $F$ is left exact. This shows that $F(e)$ is an essential monomorphism.

(ii) (1) Let $f:M\to N$ be a morphism in $\A$. Consider the morphism $F(f):F(M)\to F(N)$ in $\B$. 
Since $F(N)$ is $F(M)$-CS-Rickart, there are an essential monomorphism $e':{\rm Ker}(F(f))\to Q'$ 
and a section $s':Q'\to F(M)$ in $\mathcal{B}$ such that ${\rm ker}(F(f))=s'e'$. 
Since ${\rm Im}(F)$ is closed under direct summands, we have $Q'\cong F(Q)$ for some object $Q$ of $\A$. 
Since $F$ is left exact and full, it follows that $F({\rm ker}(f))=F(s)F(e)$ for some mophisms $e:{\rm Ker}(f)\to Q$ 
and $s:Q\to M$ in $\A$. But $F$ is also faithful, hence we have ${\rm ker}(f)=se$, where $s$ is a section. 
We claim that $e$ is an essential monomorphism. 
To this end, let $h:Q\to L$ be a morphism in $\A$ such that $he$ is a monomorphism. 
Since $F$ is left exact, $F(h)F(e)$ is a monomorphism. But $e'=F(e)$ is an essential monomorphism,
hence $F(h)$ is a monomorphism. Since $F$ is faithful, it follows that 
$h$ is a monomorphism, which shows that $e$ is an essential monomorphism.
Hence $N$ is $M$-CS-Rickart.
\end{proof}

The following result is immediate.

\begin{coll} Let $F:\mathcal{A}\to \mathcal{B}$ be an equivalence of abelian categories,
and let $M$, $N$ be objects of $\A$. Then $N$ is (dual) $M$-CS-Rickart if and only if $F(N)$ is (dual) 
$F(M)$-CS-Rickart. 
\end{coll}

\section{Transfer of (dual) CS-Rickart properties via adjoint functors}

Next we show that (dual) relative CS-Rickart properties transfer via adjoint functors under reasonable conditions. 
Let $(L,R)$ be an adjoint pair of covariant functors $L:\mathcal{A}\to \mathcal{B}$ and $R:\mathcal{B}\to \mathcal{A}$
between abelian categories. Then $L$ is right exact and $R$ is left exact. 
Denote by $\varepsilon:LR\to 1_{\mathcal{B}}$ and $\eta:1_{\mathcal{A}}\to RL$ the counit and
the unit of adjunction respectively. Following  \cite{CGW}, denote by ${\rm Stat}(R)$ the full subcategory of 
$\mathcal{B}$ 
consisting of \emph{$R$-static} objects, that is, objects $B$ of $\mathcal{B}$ such that $\varepsilon_B$ is an 
isomorphism. 
Also, denote by ${\rm Adst}(R)$ the full subcategory of $\mathcal{A}$ consisting of \emph{$R$-adstatic} objects,
that is, objects $A$ of $\mathcal{A}$ such that $\eta_A$ is an isomorphism. 

\begin{theo}\label{t:t4} Let $\mathcal{A}$ and $\mathcal{B}$ be abelian categories, 
and let $(L,R)$ be an adjoint pair of covariant functors $L:\mathcal{A}\to \mathcal{B}$ and $R:\mathcal{B}\to 
\mathcal{A}$. 
\begin{enumerate}
\item Assume that $L$ is exact, and let $M, N\in {\rm Stat}(R)$. 
\begin{enumerate}[(i)]
\item If $N$ is $M$-CS-Rickart in $\mathcal{B}$, then $R(N)$ is $R(M)$-CS-Rickart in $\mathcal{A}$. 
\item If $R$ reflects zero objects (in particular, if $R$ is faithful) and $R(N)$ is $R(M)$-CS-Rickart in 
$\mathcal{A}$, 
then $N$ is $M$-CS-Rickart in $\mathcal{B}$.
\end{enumerate}
\item Assume that $R$ is exact, and let $M, N\in {\rm Adst}(R)$. 
\begin{enumerate}[(i)]
\item If $N$ is dual $M$-CS-Rickart in $\mathcal{A}$, then $L(N)$ is dual $L(M)$-CS-Rickart in $\mathcal{B}$. 
\item If $L$ reflects zero objects (in particular, if $L$ is faithful) and $L(N)$ is dual $L(M)$-CS-Rickart in 
$\mathcal{B}$, 
then $N$ is dual $M$-CS-Rickart in $\mathcal{A}$.
\end{enumerate}
\end{enumerate}
\end{theo}

\begin{proof} (1) Denote by $\varepsilon:LR\to 1_{\mathcal{B}}$ and $\eta:1_{\mathcal{A}}\to RL$ the counit and
the unit of adjunction respectively. Then $R(\varepsilon_B)\eta_{R(B)}=1_{R(B)}$ for every object $B$ of $\mathcal{B}$,
and $\varepsilon_{L(A)}L(\eta_A)=1_{L(A)}$ for every object $A$ of $\mathcal{A}$. 

(i) Assume that $N$ is $M$-CS-Rickart in $\mathcal{B}$. 
Let $f:R(M)\to R(N)$ be a morphism in $\A$ with kernel $k={\rm ker}(f):K\to R(M)$.
Consider the morphism $g=\varepsilon_NL(f)\varepsilon_M^{-1}:M\to N$ in $\B$. 
Then there are an essential monomorphism $e:{\rm Ker}(g)\to Q$ and a section $s:Q\to M$ in $\mathcal{B}$ 
such that ${\rm ker}(g)=se$. Note that ${\rm Ker}(g)=L(K)$. 
There is a morphism $r:M\to Q$ such that $rs=1_Q$. Then $LR(s)$ is a section. 
By naturality we have the following commutative diagram in $\mathcal{A}$:
$$\SelectTips{cm}{}
\xymatrix{
 R(M) \ar[d]_{\eta_{R(M)}} \ar[r]^f & R(N) \ar[d]^{\eta_{R(N)}} \\
 RLR(M) \ar[r]_{RL(f)} & RLR(N) 
}
$$ 
Since $M,N\in {\rm Stat}(R)$, $R(\varepsilon_M)$ and $R(\varepsilon_N)$ are isomorphisms,
and so $\eta_{R(M)}=R(\varepsilon_M)^{-1}$ and $\eta_{R(N)}=R(\varepsilon_N)^{-1}$ are isomorphisms. It follows that:
$$f=\eta_{R(N)}^{-1}RL(f)\eta_{R(M)}=R(\varepsilon_N)RL(f)R(\varepsilon_M^{-1})=R(\varepsilon_N L(f) 
\varepsilon_M^{-1})=R(g).$$
Hence $k={\rm ker}(f)=R({\rm ker}(g))=R(s)R(e)$. Clearly, $R(s)$ is a section and $R(e)$ is a monomorphism. 
In order to finish the proof, it is enough to show that the monomorphism $R(e)$ is essential.

We first claim that the monomorphism $LR(e):LRL(K)\to LR(Q)$ is essential. 
By naturality we have the following commutative diagram in $\mathcal{B}$:
$$\SelectTips{cm}{}
\xymatrix{
LR(M) \ar[d]_{\varepsilon_M} \ar[r]^{LR(r)} & LR(Q) \ar[d]^{\varepsilon_Q} \ar[r]^{LR(s)} & LR(M) \ar[d]^{\varepsilon_M} 
\\
M \ar[r]_r & Q \ar[r]_s & M
}
$$
Since $\varepsilon_Q LR(r)=r\varepsilon_M$ is an epimorphism and 
$s\varepsilon_Q=\varepsilon_MLR(s)$ is a monomorphism, $\varepsilon_Q$ is an isomorphism, hence $Q\in {\rm Stat}(R)$. 
By naturality we also have the following commutative diagram in $\mathcal{B}$:
$$\SelectTips{cm}{}
\xymatrix{
LRL(K) \ar[d]_{\varepsilon_{L(K)}} \ar[r]^-{LR(e)} & LR(Q) \ar[d]^{\varepsilon_Q} \\
L(K) \ar[r]_e & Q 
}
$$
Then $e\varepsilon_{L(K)}=\varepsilon_Q LR(e)$ is a monomorphism,
hence so is $\varepsilon_{L(K)}$. But the equality $\varepsilon_{L(K)}L(\eta_K)=1_{L(K)}$ shows that  
$\varepsilon_{L(K)}$ is also a retraction, and so $\varepsilon_{L(K)}$ is an isomorphism.
Since $\varepsilon_Q LR(e)=e\varepsilon_{L(K)}$ is an essential monomorphism, so is $LR(e)$,
because a composition of monomorphisms is essential if and only if each of them is essential.

Now we claim that the monomorphism $R(e):RL(K)\to R(Q)$ is essential. 
To this end, let $h:R(Q)\to P$ be a morphism in $\mathcal{A}$ such that $hR(e)$ is a monomorphism. The functor $L$ being exact, $L(h)LR(e)$ is a monomorphism, and so is $L(h)$, 
because $LR(e)$ is an essential monomorphism. Thus $RL(h)$ is a monomorphism.
By naturality we have the following commutative diagram in $\mathcal{A}$:
$$\SelectTips{cm}{}
\xymatrix{
 R(Q) \ar[d]_{\eta_{R(Q)}} \ar[r]^h &P \ar[d]^{\eta_P} \\
 RLR(Q) \ar[r]_-{RL(h)} & RL(P) 
}
$$ 
Note that $\eta_{R(Q)}$ is an isomorphism, because so is $R(\varepsilon_Q)$. 
Since $\eta_Ph=RL(h)\eta_{R(Q)}$ is a monomorphism, so is $h$.  
This shows that $R(e)$ is an essential monomorphism. Thus $R(N)$ is $R(M)$-CS-Rickart.

(ii) Assume that $R$ reflects zero objects and $R(N)$ is $R(M)$-CS-Rickart in $\mathcal{A}$. 
Let $f:M\to N$ be a morphism in $\mathcal{B}$ with kernel 
$k={\rm ker}(f):K\to M$. Consider the morphism $R(f):R(M)\to R(N)$ in $\A$. 
Then there are an essential monomorphism $e:{\rm Ker}(R(f))\to Q$ and a section $s:Q\to R(M)$ in $\mathcal{A}$ 
such that ${\rm ker}(R(f))=se$. 

By naturality we have the following commutative diagram in $\mathcal{B}$:
$$\SelectTips{cm}{}
\xymatrix{
 LR(M) \ar[d]_{\varepsilon_M} \ar[r]^-{LR(f)} & LR(N) \ar[d]^{\varepsilon_N} \\
 M \ar[r]_f & N 
}
$$ 
in which $\varepsilon_M$ and $\varepsilon_N$ are isomorphisms. 
Since $L$ is exact, it follows that $${\rm ker}(f)={\rm ker}(LR(f))=L({\rm ker}R(f))=L(s)L(e).$$ 
Clearly, $L(s)$ is a section and $L(e)$ is a monomorphism, because $L$ is exact. 
In order to finish the proof, it is enough to show that the monomorphism $L(e)$ is essential.

We first claim that the monomorphism $RL(e):RLR(K)\to RL(Q)$ is essential. 
There is a morphism $r:R(M)\to Q$ such that $rs=1_Q$. Then $RL(s)$ is a section and $RL(r)$ is a retraction.
By naturality we have the following commutative diagram in $\mathcal{A}$:
$$\SelectTips{cm}{}
\xymatrix{
R(M) \ar[d]_{\eta_{R(M)}} \ar[r]^-r & Q \ar[d]^{\eta_Q} \ar[r]^-s & R(M) \ar[d]^{\eta_{R(M)}} \\
RLR(M) \ar[r]_-{RL(r)} & RL(Q) \ar[r]_-{RL(s)} & RLR(M)
}
$$
Since $M\in {\rm Stat}(R)$, $R(\varepsilon_M)$ is an isomorphism,
and so $\eta_{R(M)}=R(\varepsilon_M)^{-1}$ is an isomorphism. 
Since $\eta_Qr=RL(r)\eta_{R(M)}$ is an epimorphism and $RL(s)\eta_Q=\eta_{R(M)}s$ is a monomorphism,
$\eta_Q$ is an isomorphism, hence $Q\in {\rm Adst}(R)$.
By naturality we have the following commutative diagram in $\mathcal{A}$:
$$\SelectTips{cm}{}
\xymatrix{
 R(K) \ar[d]_{\eta_{R(K)}} \ar[r]^-e & Q \ar[d]^{\eta_Q} \\
 RLR(K) \ar[r]_-{RL(e)} & RL(Q) 
}
$$ 
in which $\eta_Q$ is an isomorphism and $\eta_{R(K)}$ is a section. 
Since $RL(e)\eta_{R(K)}=\eta_Q e$ is an essential monomorphism, so is $RL(e)$, because a composition of monomorphisms 
is essential if and only if each of them is essential.

Now we claim that the monomorphism $L(e):LR(K)\to L(Q)$ is essential. 
To this end, let $h:L(Q)\to P$ be a morphism in $\mathcal{B}$ such that $hL(e)$ is a monomorphism. 
Then $R(h)RL(e)$ is a monomorphism, and so $R(h)$ is a monomorphism, 
because $RL(e)$ is an essential monomorphism. We have $R({\rm Ker}(h))={\rm Ker}(R(h))=0$.  
Since $R$ reflects zero objects, it follows that ${\rm Ker}(h)=0$, and so $h$ is a monomorphism. 
This shows that $L(e)$ is an essential monomorphism. Thus $N$ is $M$-CS-Rickart.
\end{proof}

We also give the contravariant version of Theorem \ref{t:t4}, because it will be used later on. 
Let $(L,R)$ be a \emph{right adjoint pair} of contravariant functors $L:\mathcal{A}\to
\mathcal{B}$ and $R:\mathcal{B}\to \mathcal{A}$ between abelian categories \cite[45.2]{Wis}. This means that 
for every objects $A$ of $\mathcal{A}$ and $B$ of $\mathcal{B}$ there is a
bijection $$\Hom_{\mathcal{B}}(B,L(A))\cong \Hom_{\mathcal{A}}(A,R(B)),$$ natural in each variable.  
Then both $L$ and $R$ are left exact, so they turn cokernels into kernels. 
Let $\varepsilon:1_{\mathcal{B}}\to LR$ and $\eta:1_{\mathcal{A}}\to RL$ be the counit and the unit of adjunction 
respectively. 
Following \cite{Castano}, denote by ${\rm Refl}(R)$ the full subcategory of $\mathcal{B}$ consisting of 
\emph{$R$-reflexive} objects, that is, objects $B$ of $\mathcal{B}$ such that $\varepsilon_B$ is an isomorphism. 
Also, denote by ${\rm Refl}(L)$ the full subcategory of $\mathcal{A}$ consisting of \emph{$L$-reflexive} objects,
that is, objects $A$ of $\mathcal{A}$ such that $\eta_A$ is an isomorphism. 
Similarly, one may consider a \emph{left adjoint pair} of contravariant functors $L:\mathcal{A}\to
\mathcal{B}$ and $R:\mathcal{B}\to \mathcal{A}$ between abelian categories \cite[45.2]{Wis}.  
This means that for every objects $A$ of $\mathcal{A}$ and $B$ of $\mathcal{B}$ there is a
bijection $$\Hom_{\mathcal{B}}(L(A),B)\cong \Hom_{\mathcal{A}}(R(B),A),$$ natural in each variable.  
Then both $L$ and $R$ are right exact, so they turn kernels into cokernels. 

\begin{theo} \label{t:t4contrav} Let $\mathcal{A}$ and $\mathcal{B}$ be abelian categories, 
and let $(L,R)$ be a pair of contravariant functors $L:\mathcal{A}\to \mathcal{B}$ and $R:\mathcal{B}\to \mathcal{A}$. 
\begin{enumerate}
\item Assume that $(L,R)$ is left adjoint and $L$ is exact, and let $M, N\in {\rm Refl}(R)$. 
\begin{enumerate}[(i)]
\item If $N$ is $M$-CS-Rickart in $\mathcal{B}$, then $R(M)$ is dual $R(N)$-CS-Rickart in $\mathcal{A}$. 
\item If $R$ reflects zero objects (in particular, if $R$ is faithful) and $R(M)$ is dual $R(N)$-CS-Rickart in 
$\mathcal{A}$, 
then $N$ is $M$-CS-Rickart in $\mathcal{B}$.
\end{enumerate}
\item Assume that $(L,R)$ is right adjoint and $R$ is exact, and let $M, N\in {\rm Refl}(L)$. 
\begin{enumerate}[(i)]
\item If $N$ is dual $M$-CS-Rickart in $\mathcal{A}$, then $L(M)$ is $L(N)$-CS-Rickart in $\mathcal{B}$. 
\item If $L$ reflects zero objects (in particular, if $L$ is faithful) and $L(M)$ is $L(N)$-CS-Rickart in 
$\mathcal{B}$, 
then $N$ is dual $M$-CS-Rickart in $\mathcal{A}$.
\end{enumerate}
\end{enumerate}
\end{theo}

The main consequence of Theorem \ref{t:t4} is the following result.

\begin{theo} \label{t:ff} Let $\mathcal{A}$ and $\mathcal{B}$ be abelian categories, 
and let $(L,R)$ be an adjoint pair of covariant functors $L:\mathcal{A}\to \mathcal{B}$ and $R:\mathcal{B}\to 
\mathcal{A}$. 
\begin{enumerate}
\item Assume that $L$ is exact and $R$ is fully faithful. Let $M$ and $N$ be objects of $\mathcal{B}$. 
Then $N$ is $M$-CS-Rickart in $\mathcal{B}$ if and only if $R(N)$ is $R(M)$-CS-Rickart in $\mathcal{A}$. 
\item Assume that $R$ is exact and $L$ is fully faithful. Let $M$ and $N$ be objects of $\mathcal{A}$. 
Then $N$ is dual $M$-CS-Rickart in $\mathcal{A}$ if and only if $L(N)$ is dual $L(M)$-CS-Rickart in $\mathcal{B}$. 
\end{enumerate}
\end{theo}

\begin{proof} (1) Since $R$ is fully faithful, every object of $\mathcal{B}$ is in ${\rm Stat}(R)$. 
Then use Theorem \ref{t:t4}.
\end{proof}

\section{Applications}

We give several applications of our theorems, showing that their hypotheses hold in a large number of relevant 
situations. 

Let $\mathcal{A}$ be an abelian category and let $\mathcal{C}$ be a full subcategory of $\mathcal{A}$. Then
$\mathcal{C}$ is called a \emph{reflective} (\emph{coreflective}) subcategory of $\mathcal{A}$ if the inclusion functor
$i:\mathcal{C}\to \mathcal{A}$ has a left (right) adjoint. In any of the two cases, $i$ is fully faithful. If $\mathcal{C}$ is reflective, then limits (in particular, kernels) in $\mathcal{C}$ are the same as limits in 
$\mathcal{A}$, while colimits (in particular, cokernels) in $\mathcal{C}$ are taken in $\mathcal{A}$ and then reflected 
in $\mathcal{C}$ \cite[Chapter~X, Proposition~1.2]{St}. For coreflective subcategories one has the dual result. If 
$\mathcal{C}$ is a reflective (coreflective) subcategory of $\mathcal{A}$ and the left (right) adjoint of the inclusion functor $i$ preserves kernels (cokernels), then the subcategory $\mathcal{C}$ is called \emph{Giraud} (\emph{co-Giraud}). In this case, the left (right) adjoint 
of $i$ is exact.

\begin{coll} \label{c:cr} Let $\mathcal{A}$ be an abelian category, $\mathcal{C}$ an abelian full subcategory of
$\mathcal{A}$ and $i:\mathcal{C}\to \mathcal{A}$ the inclusion functor.
\begin{enumerate}[(i)]
\item Assume that $\mathcal{C}$ is a Giraud subcategory of $\mathcal{A}$. 
Let $M$ and $N$ be objects of $\mathcal{C}$. Then $N$ is $M$-CS-Rickart in $\mathcal{C}$ 
if and only if $i(N)$ is $i(M)$-CS-Rickart in $\mathcal{A}$.
\item Assume that $\mathcal{C}$ is a co-Giraud subcategory of $\mathcal{A}$. 
Let $M$ and $N$ be objects of $\mathcal{C}$. Then $N$ is dual $M$-CS-Rickart in $\mathcal{C}$ if and only if $i(N)$ is dual $i(M)$-CS-Rickart in $\mathcal{A}$.
\end{enumerate}
\end{coll}

\begin{proof} (1) Since $i$ is fully faithful and its left  adjoint is exact, one uses Theorem \ref{t:ff}.  
\end{proof}

For Grothendieck categories we have the following corollary.

\begin{coll} \label{c:gp} Let $\mathcal{A}$ be a Grothendieck category with a generator $U$ with $R={\rm
End}_{\mathcal{A}}(U)$. Let $S={\rm Hom}_{\mathcal{A}}(U,-):\mathcal{A}\to {\rm Mod}(R)$ and let
$T:{\rm Mod}(R)\to \mathcal{A}$ be a left adjoint of $S$.
Let $M$ and $N$ be objects of $\mathcal{A}$. Then $N$ is an $M$-CS-Rickart object of $\mathcal{A}$ 
if and only if $S(N)$ is an $S(M)$-CS-Rickart right $R$-module.
\end{coll}

\begin{proof} By the Gabriel-Popescu Theorem \cite[Chapter~X, Theorem~4.1]{St}, $S$ is fully faithful and 
has an exact left adjoint $T:{\rm Mod}(R)\to \mathcal{A}$. Moreover, $\A$ is equivalent to a Giraud subcategory 
of ${\rm Mod}(R)$ \cite[Chapter~X, Proposition~1.5]{St}, which is a reflective abelian full subcategory. 
Then use Corollary \ref{c:cr}.
\end{proof}

Following \cite[Section~2.2]{DNR}, let $C$ be a coalgebra over a field $k$, and let ${}^C\mathcal{M}$ be the
(Grothendieck) category of left $C$-comodules. Left $C$-comodules will be identified with rational right
$C^*$-modules, where $C^*={\rm Hom}_{k}(C,k)$. Let $i:{}^C\mathcal{M}\to {\rm Mod}(C^*)$ be the inclusion functor,
and let ${\rm Rat}:{\rm Mod}(C^*)\to {}^C\mathcal{M}$ be the functor which associates to every right $C^*$-module its 
rational $C^*$-submodule. Then $(i,{\rm Rat})$ is an adjoint pair. If $C$ is a right semiperfect coalgebra, then the  functor Rat is exact \cite[Corollary~3.2.12]{DNR}, hence ${}^C\mathcal{M}$ is a co-Giraud subcategory of ${\rm Mod}(C^*)$. Then Corollary \ref{c:cr} yields the following consequence.

\begin{coll} \label{c:com1} Let $C$ be a right semiperfect coalgebra over a field.
Let $M$ and $N$ be left $C$-comodules. Then $N$ is a dual $M$-CS-Rickart left $C$-comodule if and only if $i(N)$ is a dual $i(M)$-CS-Rickart right $C^*$-module. 
\end{coll}

Recall that an \emph{adjoint triple} of functors is a triple $(L,F,R)$ of covariant functors $F:\mathcal{A}\to 
\mathcal{B}$ and $L,R:\mathcal{B}\to \mathcal{A}$ such that $(L,F)$ and $(F,R)$ are adjoint pairs of functors. Then $F$ 
is an exact functor. It is known that $L$ is fully faithful if and only if so is $R$ \cite[Lemma~1.3]{DT}.
Now Theorem \ref{t:ff} yields the following consequence.  

\begin{coll} \label{c:tripleff} Let $(L,F,R)$ be an adjoint triple of covariant functors $F:\mathcal{A}\to \mathcal{B}$
and $L,R:\mathcal{B}\to \mathcal{A}$ between abelian categories. Let $M$ and $N$ be objects of $\mathcal{B}$, 
and assume that $L$ (or $R$) is fully faithful. Then: 
\begin{enumerate}
\item $N$ is $M$-CS-Rickart in $\mathcal{B}$ if and only if $R(N)$ is $R(M)$-CS-Rickart in $\mathcal{A}$.
\item $N$ is dual $M$-CS-Rickart in $\mathcal{B}$ if and only if $L(N)$ is dual $L(M)$-CS-Rickart in $\mathcal{A}$.
\end{enumerate}
\end{coll}

Now let $A, C$ be rings and let $_C B_A$ be a bimodule. Let $R=\begin{pmatrix} A & 0 \\ B & C \end{pmatrix}$
be the formal triangular matrix ring constructed from $A, B, C$. 
Following \cite[Chapter~4, Section A]{Goodearl}, for a right $A$-module $M$, a right $C$-module $N$ and a right
$A$-module homomorphism $f: N\otimes_C B\to M$, one may endow the abelian group $P=M\oplus N$ 
with a right $R$-module structure as follows. If we write each pair $(m, n)\in P$ in the form 
$\left (\begin{smallmatrix} & 0 \\ m & \\  & n \\ \end{smallmatrix} \right)$, scalar multiplication is defined by
$$\left (\begin{smallmatrix} & 0 \\ m & \\  & n \\ \end{smallmatrix} \right)
\begin{pmatrix} a & 0 \\ b & c \end{pmatrix}=
\left (\begin{smallmatrix}  & 0 \\ ma+f(n\otimes b) & \\  & nc \\ \end{smallmatrix} \right)$$
and we write $P=\left (\begin{smallmatrix} & 0 \\ M & \\  & N \\ \end{smallmatrix} \right)$. 
Moreover, every right $R$-module has this form. 

Following \cite[Chapter 4, Section A, Exercises~19 and 20]{Goodearl}, consider the covariant functors:
\begin{align*} & J_{23}: {\rm Mod}(C) \to {\rm Mod}(R),\quad J_{23}(N)=
 \left (\begin{smallmatrix}  & 0 \\ N\otimes_C B & \\  & N \\ \end{smallmatrix} \right), \\
& J_3: {\rm Mod}(C) \to {\rm Mod}(R),\quad J_3(N)=
J_{23}(N)/\left (\begin{smallmatrix}  & 0 \\ N\otimes_C B & \\  & 0 \\ \end{smallmatrix} \right), \\
& P_3: {\rm Mod}(R)\to {\rm Mod}(C), \quad P_3(M)=
M\left (\begin{smallmatrix} 0 & 0 \\ 0 & C \\ \end{smallmatrix} \right).
\end{align*} 
Then $(J_{23}, P_3, J_3)$ is an adjoint triple \cite[Chapter 4, Section A, Exercise 22]{Goodearl}.

\begin{coll} Consider the ring $R=\begin{pmatrix} A & 0 \\ B & C \end{pmatrix}$, 
where $A,C$ are rings and $_C B_A$ is a bimodule. Let $M$ and $N$ be right $C$-modules. Then: 
\begin{enumerate}
\item $N$ is an $M$-CS-Rickart right $C$-module if and only if $J_3(N)$ is a $J_3(M)$-CS-Rickart right $R$-module.
\item $N$ is a dual $M$-CS-Rickart right $C$-module if and only if $J_{23}(N)$ is a dual $J_{23}(M)$-CS-Rickart right 
$R$-module.
\end{enumerate}
\end{coll}

\begin{proof} Since $P_3J_{23}\cong 1_{{\rm Mod}(C)}$ and $P_3J_3\cong 1_{{\rm Mod}(C)}$ 
\cite[Chapter 4, Section A, Exercise 22]{Goodearl}, $J_{23}$ and $J_3$ are fully faithful. 
Then use Corollary \ref{c:tripleff} for the adjoint triple of functors $(J_{23}, P_3, J_3)$.
\end{proof}

Following \cite{Nasta-04}, we recall some notation and terminology on graded modules. 
In what follows $G$ will denote a group with identity element $e$, and $R$ will be a $G$-graded ring. 
For a $G$-graded ring $R=\bigoplus_{\sigma\in G}R_{\sigma}$, denote by ${\rm gr}(R)$ the (Grothendieck) category which
has as objects the $G$-graded unital right $R$-modules and as morphisms the morphisms of $G$-graded unital right
$R$-modules. We consider the following functors:
\begin{enumerate}
\item For $\sigma\in G$, the \emph{$\sigma$-suspension functor} $T_{\sigma}:\mathrm{gr}(R)\to \mathrm{gr}(R)$ defined 
by 
$T_{\sigma}(M)=M({\sigma})$. Recall that for a graded right $R$-module $M$ and $\sigma\in G$, 
the \emph{$\sigma$-suspension} $M(\sigma)$ of $M$ 
is the graded right $R$-module which is equal to $M$ as a right $R$-module and the gradation is given by 
$M({\sigma})_{\tau}=M_{\sigma\tau}$ for every $\tau\in G$.  
\item For $\sigma\in G$, the functor $(-)_{\sigma}:\mathrm{gr}(R)\to {\rm Mod}(R_e)$, 
which associates to every graded right $R$-module $M=\bigoplus_{\tau\in G}M_{\tau}$ the right $R_e$-module 
$M_{\sigma}$. 
\item The \emph{induced functor} ${\rm Ind}:{\rm Mod}(R_e)\to {\rm gr}(R)$ defined as follows: 
for a right $R_e$-module $N$, ${\rm Ind}(N)$ is the graded right $R$-module $M=N\otimes_{R_e}R$, 
where the gradation of $M=\bigoplus_{\sigma\in G}M_{\sigma}$ is given by 
$M_{\sigma}=N_{\sigma}\otimes_{R_e}R$ for every $\sigma\in G$. 
\item The \emph{coinduced functor} ${\rm Coind}:{\rm Mod}(R_e)\to {\rm gr}(R)$ defined as follows:
for a right $R_e$-module $N$, ${\rm Coind}(N)$ is the graded right $R$-module 
$M^*=\bigoplus_{\sigma\in G}M'_{\sigma}$, where 
\[M'_{\sigma}=\{f\in {\rm Hom}_{R_e}(R,N)\mid f(R_{\sigma'})=0 \textrm{ for every } \sigma'\neq \sigma^{-1}\}.\]
\end{enumerate}

\begin{coll} \label{c1:triplegr} Let $R$ be a $G$-graded ring, and let $M$ and $N$ be right $R_e$-modules. Then: 
\begin{enumerate} 
\item $N$ is an $M$-CS-Rickart right $R_e$-module if and only if 
${\rm Coind}(N)$ is a ${\rm Coind}(M)$-CS-Rickart graded right $R$-module.
\item $N$ is a dual $M$-CS-Rickart right $R_e$-module if and only if 
${\rm Ind}(N)$ is a dual ${\rm Ind}(M)$-CS-Rickart graded right $R$-module.
\end{enumerate}
\end{coll}

\begin{proof} By \cite[Theorem~2.5.5]{Nasta-04}, for $\sigma\in G$, 
$(T_{\sigma^{-1}}\circ {\rm Ind},(-)_{\sigma},T_{\sigma^{-1}}\circ {\rm Coind})$ is an adjoint triple, 
and we have $(-)_{\sigma}\circ T_{\sigma^{-1}}\circ {\rm Ind}\cong 1_{{\rm Mod}(R_e)}$ 
and $(-)_{\sigma}\circ T_{\sigma^{-1}}\circ {\rm Coind}\cong 1_{{\rm Mod}(R_e)}$, that is,  
the functors $T_{\sigma^{-1}}\circ {\rm Ind}$ and $T_{\sigma^{-1}}\circ {\rm Coind}$ are fully faithful.
In particular, $({\rm Ind},(-)_{e},{\rm Coind})$ is an adjoint triple,
and ${\rm Ind}$ and ${\rm Coind}$ are fully faithful.
Now use Corollary \ref{c:tripleff}. 
\end{proof}

Now let us recall the concept of recollement of abelian categories, following \cite{Ps}. For an additive functor 
$F:\A\to \B$ between abelian categories, we denote by ${\rm Ker}(F)$ the kernel of $F$, which consists of all objects 
$A$ of $\A$ such that $F(A)=0$. A \emph{recollement} of abelian categories $\A$, $\B$ and $\C$, denoted by 
$(\A,\B,\C)$, 
is a diagram of functors: 
$$\SelectTips{cm}{}
\xymatrix{
 \A \ar[rr]^i && \B \ar[rr]^e \ar@/^1.5pc/[ll]^{p} \ar@/_1.5pc/[ll]_{q} && \C \ar@/^1.5pc/[ll]^{r} \ar@/_1.5pc/[ll]_{l}
}
$$
which satisfy the following conditions: (i) $(l,e,r)$ is an adjoint triple; (ii) $(q,i,p)$ is 
an adjoint triple; (iii) $i$, $l$ and $r$ are fully faithful; (iv) ${\rm Im}(i)={\rm Ker}(e)$. 

Now Corollary \ref{c:tripleff} yields the following consequence.

\begin{coll} \label{c:rec} Let $(\A,\B,\C)$ be a recollement of abelian categories as above. Let $M$ and $N$ be objects of $\mathcal{C}$. Then:
\begin{enumerate}[(1)]
\item $N$ is $M$-CS-Rickart in $\mathcal{C}$ if and only if $r(N)$ is $r(M)$-CS-Rickart in $\mathcal{B}$.
\item $N$ is dual $M$-CS-Rickart in $\mathcal{C}$ if and only if $l(N)$ is dual $l(M)$-CS-Rickart in $\mathcal{B}$. 
\end{enumerate}
\end{coll}

We recall a specific example of recollement of abelian categories from \cite{Ps}, where further relevant examples are 
presented. Let $R$ be a unitary ring and let $e=e^2\in R$ be an idempotent. Then the following diagram is a recollement 
of abelian categories \cite[Example~2.7]{Ps}:
$$\SelectTips{cm}{}
\xymatrix{
 {\rm Mod}(R/ReR) \ar[rr]^i && {\rm Mod}(R) \ar[rr]^{e(-)} \ar@/^1.5pc/[ll]^{{\rm Hom}_R(R/ReR,-)} 
\ar@/_1.5pc/[ll]_{(R/ReR)\otimes_R-} && {\rm Mod}(eRe) \ar@/^1.5pc/[ll]^{{\rm Hom}_{eRe}(eR,-)} 
\ar@/_1.5pc/[ll]_{Re\otimes_{eRe}-}
}
$$
where the module categories are categories of left modules,  $i$ is the inclusion functor and $e(-)$ is the functor given by left multiplication by $e$. Now we have the following consequence of Corollary \ref{c:rec}.

\begin{coll} \label{c:rec2} Let $R$ be a unitary ring and let $e=e^2\in R$ be an idempotent. Let $M$ and $N$ be objects of ${\rm Mod}(eRe)$. Then: 
\begin{enumerate}[(1)]
\item $N$ is $M$-CS-Rickart in ${\rm Mod}(eRe)$ if and only if ${\rm Hom}_{eRe}(eR,N)$ is ${\rm 
Hom}_{eRe}(eR,M)$-CS-Rickart in ${\rm Mod}(R)$.
\item $N$ is dual $M$-CS-Rickart in ${\rm Mod}(eRe)$ if and only if $Re\otimes_{eRe}N$ is dual 
$Re\otimes_{eRe}M$-CS-Rickart in ${\rm Mod}(R)$.
\end{enumerate}
\end{coll}

\section{Endomorphism rings}

In this section we give some applications to endomorphism rings of (graded) modules and comodules. 
We start with a specific result for modules, for which we need the following concepts.

\begin{defn} \rm A right $R$-module $M$ is called: 
\begin{enumerate}[(i)]
\item \emph{im-local-retractable} if for every monomorphism $k:K\to M$ and for every $x\in K$, 
there exists a homomorphism $h:M\to K$ such that $x\in {\rm Im}(hk)$. 
\item \emph{im-local-coretractable} if for every epimorphism $c:M\to C$ and for every $z\in C$, 
there exists a homomorphism $h:C\to M$ such that $z\in {\rm Im}(ch)$. 
\end{enumerate}
\end{defn}

\begin{rem} \label{r:kc} \rm (i) Clearly, every im-local-retractable right $R$-module $M$ is $k$-local-retractable,
in the sense that for every endomorphism $f:M\to M$ with $k={\rm ker}(f):K\to M$ and for every $x\in K$, 
there exists a homomorphism $h:M\to K$ such that $x\in {\rm Im}(hk)$ \cite[Definition~3.6]{LRR10}. 
Dually, every im-local-coretractable right $R$-module $M$ is $c$-local-coretractable, 
in the sense that for every endomorphism $f:M\to M$ with $c={\rm coker}(f):M\to C$ and for every $z\in C$, 
there exists a homomorphism $h:C\to M$ such that $z\in {\rm Im}(ch)$ \cite[Definition~4.5]{CKT20-2}.

(ii) The class of im-local-retractable right $R$-modules and the class of im-local-coretractable right $R$-modules
are closed under direct summands. Indeed, for the former, let $M$ be a im-local-retractable right $R$-module
and let $N$ be a direct summand of $M$. Let $k:K\to N$ be a monomorphism, $i:N\to M$ the canonical inclusion and $x\in 
K$. 
Then there is a homomorphism $h:M\to K$, and so there is the homomorphism $hi:N\to K$ 
such that $x\in {\rm Im}(hik)$. Hence $N$ is im-local-retractable.
\end{rem}

\begin{coll} \label{c:end0} Let $M$ be a right $R$-module, and let $S={\rm End}_R(M)$.
\begin{enumerate}
\item If $M$ is im-local-retractable and $S$ is a self-CS-Rickart right $S$-module, 
then $M$ is a self-CS-Rickart right $R$-module.
\item If $M$ is im-local-coretractable and $S$ is a dual self-CS-Rickart left $S$-module,
then $M$ is a dual self-CS-Rickart right $R$-module.
\end{enumerate}
\end{coll}

\begin{proof} (1) Consider the adjoint pair of functors $(T,H)$, 
where $$T=-\otimes_SM:{\rm Mod}(S)\to {\rm Mod}(R),$$ $$H={\rm Hom}_R(M,-):{\rm Mod}(R)\to {\rm Mod}(S).$$
Assume that $M$ is im-local-retractable and $S$ is a self-CS-Rickart right $S$-module.
Let $f:M\to M$ be a homomorphism with kernel $k={\rm ker}(f):K\to M$. 
Since $H(M)=S$ is a self-CS-Rickart right $S$-module, $H(k)={\rm ker}(H(f))=sm$ for some essential monomorphism 
$m:H(K)\to Q$ and section $s:Q\to H(M)$. Let $r:H(M)\to Q$ be a homomorphism such that $rs=1_{Q}$.
Denote $t=\varepsilon_M T(s):T(Q)\to M$ and $e=T(r)\varepsilon_M^{-1}k:K\to T(Q)$. 
Then $t$ is a section.

We prove that $k=te$. To this end, let $x\in K$. 
Since $M$ is im-local-retractable, there exists a homomorphism $h:M\to K$ such that $x=hk(y)$ for some $y\in K$. 
By naturality we have the following commutative diagram:
$$\SelectTips{cm}{}
\xymatrix{
 TH(M) \ar[d]_{\varepsilon_M} \ar[r]^-{TH(kh)} & TH(M) \ar[d]^{\varepsilon_M} \\
 M \ar[r]_{kh} & M
}
$$  
It follows that:
\begin{align*}
\begin{split}
te(x)&=\varepsilon_M T(s)T(r)\varepsilon_M^{-1}k(x) =\varepsilon_M T(s)T(r)\varepsilon_M^{-1}khk(y) \\
&=\varepsilon_M T(s)T(r)\varepsilon_M^{-1}kh\varepsilon_M \varepsilon_M^{-1}k(y) 
=\varepsilon_M T(s)T(r)\varepsilon_M^{-1} \varepsilon_M TH(kh) \varepsilon_M^{-1}k(y) \\
&=\varepsilon_M T(s)T(r) TH(kh) \varepsilon_M^{-1}k(y) 
=\varepsilon_M T(s)T(r) TH(k)TH(h) \varepsilon_M^{-1}k(y) \\
&=\varepsilon_M T(s)T(r) T(s)T(m)TH(h) \varepsilon_M^{-1}k(y) 
=\varepsilon_M T(s)T(m)TH(h) \varepsilon_M^{-1}k(y) \\
&=\varepsilon_M TH(k)TH(h) \varepsilon_M^{-1}k(y) 
=\varepsilon_M TH(kh) \varepsilon_M^{-1}k(y) \\
&=kh \varepsilon_M \varepsilon_M^{-1}k(y)=khk(y)=k(x).
\end{split}
\end{align*}
Hence $k=te$, which also shows that $e$ is a monomorphism. 

Finally, let us prove that $e$ is an essential monomorphism. To this end, let $v:T(Q)\to V$ be a homomorphism
such that $ve:K\to V$ is a monomorphism. 
As in the proof of Theorem \ref{t:t4} (ii), $\eta_{Q}:Q\to HT(Q)$ is an isomorphism. Then we have:
\begin{align*} H(ve)& =H(v)HT(r)H(\varepsilon_M^{-1})H(k)=H(v)HT(r)\eta_{H(M)}H(k) \\
&=H(v)\eta_{Q}rH(k)=H(v)\eta_{Q}rsm=H(v)\eta_{Q}m.
\end{align*}
and this is a monomorphism. Since $\eta_{Q}m$ is an essential monomorphism, $H(v)$ must be a monomorphism. 
Let $a={\rm ker}(v):A\to T(Q)$ and $z\in A$. Since $M$ is im-local-retractable, so is its direct summand $T(Q)$ 
by Remark \ref{r:kc}. Hence there exists a homomorphism $g:T(Q)\to A$ such that $z=ga(u)$ for some $u\in A$.
By naturality we have the following commutative diagram:
$$\SelectTips{cm}{}
\xymatrix{
 THT(Q) \ar[d]_{\varepsilon_{T(Q)}} \ar[r]^-{TH(ag)} & THT(Q) \ar[d]^{\varepsilon_{T(Q)}} \\
 T(Q) \ar[r]_{ag} & T(Q)
}
$$ 
Since $H(a)={\rm ker}(H(v))=0$, it follows that:
$$a(z)=aga(u)=ag\varepsilon_{T(Q)}\varepsilon_{T(Q)}^{-1}a(u)=\varepsilon_{T(Q)}TH(ag)\varepsilon_{T(Q)}^{-1}a(u)=0.$$ 
Hence ${\rm ker}(v)=0$, and so $v$ is a monomorphism. This shows that $e$ is an essential monomorphism,
and thus $M$ is a self-CS-Rickart right $R$-module.
\end{proof}

Following \cite{Ishi}, a covariant functor $F:\A\to \B$ between abelian categories 
is called \emph{faithfully exact} provided the sequence 
$F(A)\stackrel{F(f)}\longrightarrow F(B)\stackrel{F(g)}\longrightarrow F(C)$ is exact if and only if 
the sequence $A\stackrel{f}\longrightarrow B\stackrel{g}\longrightarrow C$ is exact.
Let ${}_SM_R$ be a bimodule. The right $R$-module $M$ is called \emph{faithfully projective} if the functor 
${\rm Hom}_R(M,-):{\rm Mod}(R)\to {\rm Mod}(S)$ is faithfully exact. 
The right $R$-module $M$ is called \emph{faithfully injective} if the functor
${\rm Hom}_R(-,M):{\rm Mod}(R)\to {\rm Mod}(S^{\rm op})$ is faithfully exact.
The left $S$-module $M$ is called \emph{faithfully flat} if the functor 
$-\otimes_SM:{\rm Mod}(S)\to {\rm Mod}(R)$ is faithfully exact. 

\begin{coll} \label{c:end} Let $M$ be a right $R$-module and $S={\rm End}_R(M)$. 
\begin{enumerate}
\item Assume that $M$ is a flat left $S$-module. 
\begin{enumerate}[(i)]
\item If $M$ is a self-$CS$-Rickart right $R$-module, then $S$ is a self-$CS$-Rickart right $S$-module.
\item If $M$ is a faithfully projective right $R$-module and $S$ is a self-$CS$-Rickart right $S$-module, 
then $M$ is a self-$CS$-Rickart right $R$-module.
\end{enumerate}
\item Assume that $M$ is a projective right $R$-module. 
\begin{enumerate}[(i)]
\item If $S$ is a dual self-$CS$-Rickart right $S$-module, then $M$ is a dual self-$CS$-Rickart right $R$-module.
\item If $M$ is a faithfully flat left $S$-module and $M$ is a dual self-$CS$-Rickart right $R$-module, 
then $S$ is a dual self-$CS$-Rickart right $S$-module.
\end{enumerate}
\item Assume that $M$ is an injective left $S$-module. 
\begin{enumerate}[(i)]
\item If $M$ is a dual self-CS-Rickart right $R$-module, then $S$ is a self-CS-Rickart left $S$-module. 
\item If $M$ is a faithfully injective right $R$-module and $S$ is a self-CS-Rickart left $S$-module, 
then $M$ is a dual self-CS-Rickart right $R$-module.
\end{enumerate}
\end{enumerate}
\end{coll}

\begin{proof} Consider the adjoint pair of covariant functors $(T,H)$, 
where $$T=-\otimes_SM:{\rm Mod}(S)\to {\rm Mod}(R),$$ $$H={\rm Hom}_R(M,-):{\rm Mod}(R)\to {\rm Mod}(S),$$
and the right adjoint pair of contravariant functors $(H_1,H_2)$,
where $$H_1={\rm Hom}_R(-,M):{\rm Mod}(R)\to {\rm Mod}(S^{\rm op}),$$ $$H_2={\rm Hom}_S(-,M):{\rm Mod}(S^{\rm op})\to 
{\rm
Mod}(R).$$

(1) We have $TH(M)\cong M$, whence $M\in {\rm Stat}(H)$. Since $M$ is a flat left $S$-module, $T$ is exact.
If $M$ is a faithfully projective right $R$-module, then $H$ reflects zero objects \cite[Theorem~2.2]{Ishi}.
Finally, use Theorem \ref{t:t4} (1).

(2) We have $HT(S)\cong S$, whence $S\in {\rm Adst}(H)$. Since $M$ is a projective right $R$-module, $H$ is exact.
If $M$ is a faithfully flat left $S$-module, then $T$ reflects zero objects \cite[Theorem~2.1]{Ishi}.
Finally, use Theorem \ref{t:t4} (2).

(3) We have $H_2H_1(M)\cong M$, whence $M\in {\rm Refl}(H_1)$. Since $M$ is an injective left $S$-module, $H_2$ is 
exact. 
If $M$ is a faithfully injective right $R$-module, then $H_1$ reflects zero objects \cite[Theorem~3.1]{Ishi}. 
Finally, use Theorem \ref{t:t4contrav} (2).
\end{proof}

Following \cite{Nasta-04}, for a $G$-graded ring $R$ and two graded right $R$-modules $M$ and $N$, 
we may consider the graded abelian group $\HOM{R}{M}{N}$, whose
$\sigma$-th homogeneous component is $$\HOM{R}{M}{N}_{\sigma} = \{f\in {\rm Hom}_R(M,N) \,\mid\, f(M_{\lambda})\subseteq
N_{\lambda\sigma} \mbox{ for all }\lambda\in G \}.$$ Then $S={\rm END}_R(M)=\HOM{R}{M}{M}$ is a $G$-graded ring
(where the multiplication is the map composition) 
and $M$ is a graded $(S,R)$-bimodule, that is, $S_{\tau}\cdot M_{\sigma}\cdot R_{\lambda}\subseteq
M_{\tau\sigma\lambda}$ for every $\tau,\sigma,\lambda\in G$. 
For a graded right $S$-module $N$, the right $R$-module $N\otimes_SM$ may be graded by
$$(N\otimes_SM)_{\tau}=\left \{\sum_{\sigma\lambda=\tau}n_{\sigma}\otimes m_{\lambda} \mid n_{\sigma}\in
N_{\sigma}, m_{\lambda}\in M_{\lambda}\right \}.$$ 

Following the terminology from modules, one may naturally define the concepts of 
\emph{faithfully flat}, \emph{faithfully projective} and \emph{faithfully injective} graded modules.

\begin{coll} \label{c:endgr} Let $M$ be a graded right $R$-module and $S={\rm END}_R(M)$. 
\begin{enumerate}
\item Assume that $M$ is a flat graded left $S$-module. 
\begin{enumerate}[(i)]
\item If $M$ is a self-$CS$-Rickart graded right $R$-module, then $S$ is a self-$CS$-Rickart graded right $S$-module.
\item If $M$ is a faithfully projective graded right $R$-module and $S$ is a self-$CS$-Rickart graded right $S$-module, 
then $M$ is a self-$CS$-Rickart graded right $R$-module.
\end{enumerate}
\item Assume that $M$ is a projective graded right $R$-module. 
\begin{enumerate}[(i)]
\item If $S$ is a dual self-$CS$-Rickart graded right $S$-module, then $M$ is a dual self-$CS$-Rickart graded right 
$R$-module.
\item If $M$ is a faithfully flat graded left $S$-module and $M$ is a dual self-$CS$-Rickart graded right $R$-module, 
then $S$ is a dual self-$CS$-Rickart graded right $S$-module.
\end{enumerate}
\item Assume that $M$ is an injective graded left $S$-module. 
\begin{enumerate}[(i)]
\item If $M$ is a dual self-CS-Rickart graded right $R$-module, then $S$ is a self-CS-Rickart graded left $S$-module. 
\item If $M$ is a faithfully injective graded right $R$-module and $S$ is a self-CS-Rickart graded left $S$-module, 
then $M$ is a dual self-CS-Rickart graded right $R$-module.
\end{enumerate}
\end{enumerate}
\end{coll}

\begin{proof} Consider the adjoint pair of covariant functors $(T,H)$, where $$T=-\otimes_SM:{\rm gr}(S)\to {\rm
gr}(R),$$ $$H=\HOM{R}{M}{-}:{\rm gr}(R)\to {\rm gr}(S),$$
and the right adjoint pair of contravariant functors $(H_1,H_2)$, where $$H_1=\HOM{R}{-}{M}:{\rm gr}(R)\to
{\rm gr}(S^{\rm op}),$$ $$H_2=\HOM{S}{-}{M}:{\rm gr}(S^{\rm op})\to {\rm gr}(R).$$

(1) We have $TH(M)\cong M$, whence $M\in {\rm Stat}(H)$. Since $M$ is a flat graded left $S$-module, $T$ is exact.
If $M$ is a faithfully projective graded right $R$-module, then $H$ reflects zero objects \cite[Theorem~1.1]{Ishi}.
Finally, use Theorem \ref{t:t4} (1).

(2) We have $HT(S)\cong S$, whence $S\in {\rm Adst}(H)$. Since $M$ is a projective graded right $R$-module, $H$ is 
exact.
If $M$ is a faithfully flat graded left $S$-module, then $T$ reflects zero objects \cite[Theorem~1.1]{Ishi}.
Finally, use Theorem \ref{t:t4} (2).

(3) We have $H_2H_1(M)\cong M$, whence $M\in {\rm Refl}(H_1)$. Since $M$ is an injective graded left $S$-module, $H_2$ 
is exact. 
If $M$ is a faithfully injective graded right $R$-module, then $H_1$ reflects zero objects \cite[Theorem~3.1]{Ishi}. 
Finally, use Theorem \ref{t:t4contrav} (2).
\end{proof}

Let $C$ be a coalgebra over a field. A left $C$-comodule $Q$ is called \emph{quasi-finite} if 
$\Hom_C(M,Q)$ is finite dimensional for every finite dimensional left $C$-comodule $M$ \cite[Definition~1.1]{Tak}. 
For a $C$-$D$-bicomodule $Q$, consider the \emph{cotensor functor} $-\square_CQ:\mathcal{M}^C\to \mathcal{M}^D$ 
(see \cite[p.~87]{DNR}). If $Q$ is quasi-finite as a right $D$-comodule, then the cotensor functor has 
a left adjoint denoted by $h_D(Q,-):\mathcal{M}^D\to \mathcal{M}^C$ \cite[Proposition~1.10]{Tak}, 
and called the \emph{cohom functor}.

\begin{coll} \label{c:endcom} Let $D$ be a coalgebra over a field, let $Q$ be a quasi-finite
injective right $D$-comodule, and let $C={\rm h}_D(Q,Q)$. Then:
\begin{enumerate}[(i)]
\item If $Q$ is a dual self-CS-Rickart right $D$-comodule, then $C$ is a dual self-CS-Rickart right $C$-comodule. 
\item If $h_D(Q,-)$ reflects zero objects and $C$ is a dual self-CS-Rickart right $C$-comodule, 
then $Q$ is a dual self-CS-Rickart right $D$-comodule.
\end{enumerate}
\end{coll}

\begin{proof} Consider the adjoint pair of covariant functors $(L,R)$, where 
$$L=h_D(Q,-):\mathcal{M}^D\to \mathcal{M}^C,$$ $$R=-\square_CQ:\mathcal{M}^C\to \mathcal{M}^D.$$ 
We have $RL(Q)\cong Q$, whence $Q\in {\rm Adst}(R)$. Since $Q$ is an injective right $D$-comodule, 
$R$ is exact \cite[Theorem~2.4.17]{DNR}. Finally, use Theorem \ref{t:t4} (2).
\end{proof}

\end{document}